\documentclass{amsart}
\usepackage[hidelinks]{hyperref}
\usepackage{PaperStyle}
\usepackage{microtype}

\usepackage{amsrefs}

\author{Benjamin B. McMillan}
\title{}
\hypersetup{
 pdfauthor={Benjamin B. McMillan},
 pdftitle={},
 pdfkeywords={},
 pdfsubject={},
 pdflang={English}}
\begin{document}

\thanks{This work was supported by the Institute for Basic Science (IBS-R032-D1).}
\address{Center For Complex Geometry, Institute for Basic Science, 55 Expo-ro, Yuseong-gu 34126 Daejeon, South Korea}
\email{mcmillan@ibs.re.kr}

\title{Foliation surgeries and the concordance groups of foliated spheres}

\begin{abstract}
  We define a general procedure extending surgery to manifolds with foliation or Haefliger structure. We find a single obstruction to foliation surgery along an attaching sphere. When unobstructed, the surgery can be chosen to preserve characteristic numbers.

  Studying these obstructions, we obtain two results. First, on every stably trivial manifold of dimension \( n \le 2q+2 \), a transversely framed codimension-\( q \) Haefliger structure can be surgered to a Haefliger structure on the sphere \( S^{n} \), characteristic numbers unchanged.
  As an application, we modify a construction of Thurston's to give explicit Haefliger structures on \( S^{2q+1} \) whose Godbillon-Vey numbers surject to the real line.
  Second, the foliation connected sum gives, for each dimension \( n \) and codimension \( q \), a group structure on concordance classes of transversely oriented Haefliger structures on \( S^{n} \). These groups may be identified with the homotopy groups of the Haefliger classifying spaces \( B\Gamma^{+}_{q} \).
 
\end{abstract}

\maketitle
\tableofcontents

\section{Introduction}
\subsection{Foliation surgery}
Surgery of smooth manifolds along spheres is among the most fundamental tools used to study the differential topology of manifolds.
With the intention of studying the differential topology of foliations, it is the aim of this paper to define a general procedure extending surgery to manifolds carrying a foliation, or more generally a Haefliger structure, obtaining a new Haefliger structure on the surgered manifold.
The surgered foliations so obtained can be well controlled, having, for example, the desirable property that their characteristic numbers are equal to those of the original.


Recall that given a smooth closed manifold \( M \), and an embedded attaching sphere \( \varphi_{0} \colon S^{k} \to M \) with trivialized normal bundle, the surgery along \( \varphi_{0} \) is a new manifold \( M' \) with slightly modified topology.
It furthermore holds that the union \( M \cup M' \) may be realized as the boundary of a simple cobordism \( W \), the trace of the surgery.
A key feature is that the manifolds \( M \) and \( M' \) can be identified outside of the domain of surgery.
As such, if \( M \) has a Haefliger structure \( \F \), one may ask whether \( M' \) admits a Haefliger structure \( \F' \) that agrees with \( \F \) outside of the domain of surgery.
In case that such an \( \F' \) exists, we say that it is a \emph{weak foliation surgery}.
We say that \( \F' \) is a \emph{strong foliation surgery} if there furthermore exists an extension of the boundary foliation \( (M, \F) \cup (M', \F') \) to all of \( W \).
The strong surgeries have the property that for any foliation characteristic class of degree equal the dimension of \( M \), the associated characteristic numbers of \( (M, \F) \) and \( (M', \F') \) are equal.

\subsection{The concordance class group}

In the case of surgery along a \( 0 \)-sphere, one can always construct a strong foliation surgery (cf. Proposition \ref{thm: prop 0 surgeries exist}).
It furthermore holds that after a choice of transverse orientation on \( \F \), the operation is well defined up to concordance (Lemma \ref{thm: lem 0-surgery well defined}).
This allows the straightforward definition, for each codimension \( q \) and round sphere \( S^{k} \), of a group 
\[ \Theta^{k}_{\Gamma^{+}_{q}} := \operatorname{Fol}^{+}_{q}(S^{k}) , \]
the set of concordance classes of transversely oriented codimension-\( q \) Haefliger structures on \( S^{k} \), with group law given by foliation connected sum (a special case of foliation \( 0 \)-surgery).
The group so obtained is naturally identified with the fundamental group \( \pi_{k}(B\Gamma^{+}_{q}) \) of the Haefliger classifying space for transversely oriented codimension-\( q \) foliations, which is of clear interest in the theory of foliations.
This identification has useful immediate consequences, and is essential to our analysis of foliation surgeries on attaching spheres of positive dimension.

\subsection{Obstructions to foliation surgery}

The situation is more interesting on the attaching spheres \( \varphi_{0} \colon S^{k} \to M \) of positive dimension, as one cannot in all cases obtain a foliation surgery; the restriction of \( \F \) off of the domain of surgery may not be extendable to any foliation \( \F' \) on \( M' \).
The obstruction to doing so is provided by the pullback of \( \F \) to \( S^{k} \) along \( \varphi_{0} \), or more strictly speaking, its concordance class.
Indeed, it is not difficult to see that if a weak surgery exists, then \( \varphi_{0}^{*}\F \) is in the trivial concordance class of foliations on \( S^{k} \).
We show in Theorem \ref{thm: unobstructed foliation surgery exists} that this is the only obstruction: if \( \varphi_{0}^{*}\F \) is concordant to the trivial Haefliger structure, then a strong foliation surgery along \( \varphi_{0} \) can be made.

These obstructions can be handled in some cases, and in particular for spheres of small dimension \( k \) relative to the foliation codimension \( q \).
To describe this, we pause to recall that the classifying space \( B\Gamma_{q} \) for codimension-\( q \) Haefliger structures \cite{Haefliger1970--FeuilletagesSurLesVarietesOuvertes} is related to the classifying space of vector bundles by the universal normal bundle map \( \nu \colon B\Gamma_{q} \to \GL_{q} \), composition with which sends the classifying map of a Haefliger structure to the classifying map of its normal bundle.
By taking the homotopy fiber of \( \nu \), one obtains the space \( F\Gamma_{q} \), which classifies Haefliger structures with trivial normal bundle.
These spaces can be fit into the fibration sequence
\[ F\Gamma_{q} \to B\Gamma_{q} \to \BGL_{q} , \]
and for transversely oriented foliations, is the fibration sequence
\[ F\Gamma_{q} \to B\Gamma^{+}_{q} \to \BGL^{+}_{q} . \]
It is well known, by the Mather-Thurston Theorem \cite{Thurston1974--FoliationsGroupsDiffeomorphisms}, that \( F\Gamma_{q} \) is \( (q+1) \)-connected, and so from the fibration sequence that \( \pi_{k}(B\Gamma^{+}_{q}) = \pi_{k}(\BGL^{+}_{q}) \) for \( k \le q+1 \).
Interpreted as a statement about Haefliger structures on \( S^{k} \), this means that the normal bundle is the precise obstruction to being null-concordant for these \( k \).

\subsection{Surgeries to the sphere}

As a consequence of this observation, if one starts with a codimension-\( q \) Haefliger structure \( (M, \F) \) with trivial normal bundle, then any attaching sphere of dimension up to \( q+1 \) has vanishing foliation surgery obstruction.
We furthermore show in Theorem \ref{thm: unobstructed foliation surgery exists} that the foliation surgery in this case can be chosen so that the surgered Haefliger structure \( (M',\F') \) has again trivial normal bundle.
This sets us up nicely to make inductive sequences of foliation surgeries, removing much of the lower dimensional homotopy groups of \( M \) in the process.
In the particular case of a closed manifold with stably trivial tangent bundle, it is known \cites{Milnor1961--ProcedureKillingHomotopyGroupsDifferentiableManifolds, Milnor1960--CobordismRingOmegaComplexAnalogue, Wall1959--NoteCobordismRing} that one can find a sequence of manifold surgeries that inductively reduce the homotopy groups, the surgery sequence terminating at the sphere.
By repeatedly applying Theorem \ref{thm: unobstructed foliation surgery exists} to this sequence of surgeries, we have thus the following (compare Theorem \ref{thm: foliations surgerable to spheres}).
\begin{theorem}\label{thm: thm intro surgery to sphere}
  Given a stably trivial manifold \( M \) of dimension \( n \) at most \( 2q+2 \), and a codimension-\( q \) Haefliger structure \( (M, \F) \) with trivial normal bundle, there exists a sequence of strong foliation surgeries to the sphere \( (S^{n}, \F') \).
  The characteristic numbers of \( \F \) and \( \F' \) agree.
\end{theorem}
Theorem \ref{thm: thm intro surgery to sphere} can be interpreted as a statement about the topology of \( F\Gamma_{q} \), as follows.
A trivially-normal Haefliger structure \( (M, \F) \) as in the Theorem represents a homology class 
\[ h_{*}([M]) \in H_{n}(F\Gamma_{q}) , \]
the pushforward of the fundamental class of \( M \) by the classifying map of \( \F \),
\[ h \colon M \to F\Gamma_{q} . \]
In the case that \( M \) is stably trivial, the associated foliation \( (S^{n}, \F') \) guaranteed by the Theorem determines a homotopy class
\[ [h' \colon S^{n} \to F\Gamma_{q}] \in \pi_{n}(F\Gamma_{q}) .\]
In fact, for non-torsion homology classes (such as those detected by foliation characteristic classes), the requirement that \( M \) be stably trivial is not crucial, because \( M \) is dominated by some stably trivial manifold \( \hat{M} \).
By using a smooth positive degree map \( f \colon \hat{M} \to M \) to pull back \( \F \), one obtains a foliation \( \hat{\F} \) on \( \hat{M} \) with all characteristic numbers scaled by the degree of \( f \).
The upshot is that if two homology classes in \( H_{n}(F\Gamma_{q}) \), say represented by \( (M, \F_{1}), (M, \F_{2}) \), are distinguished by characteristic classes, then so are \( (\hat{M}, \hat{\F}_{1}), (\hat{M}, \hat{\F}_{2}) \), and applying the Theorem, so are the associated elements of \( \pi_{n}(F\Gamma_{q}) \).

\subsection{Application to independently varying characteristic classes}

This conclusion is particularly interesting in dimension \( 2q+1 \), where large families of non-concordant foliations, and hence homology classes in \( B\Gamma_{q} \), are known to exist.
These homology classes are distinguished by evaluation against degree-\( (2q+1) \) Gel'fand-Fuks characteristic classes, which define surjective maps 
\begin{equation}\label{eq: characteristic map surjection}
  H_{2q+1}(B\Gamma_{q}) \to V_{q}
\end{equation}
to positive dimensional vector spaces \( V_{q} \).

The first example of such continuously varying foliation characteristic classes was Thurston's \cite{Thurston1972--NoncobordantFoliationsS3}, where he constructs codimension-\( 1 \) foliations on \( S^{3} \) whose Godbillon-Vey numbers give a surjection to \( \R \).
Because they are defined on the sphere, these foliations give the stronger conclusion, a surjection
\[ \pi_{3}(F\Gamma_{1}) \to \R . \]
Many examples of continuous variation were constructed later, by authors such as Kamber and Tondeur \cite{Kamber-Tondeur1974--CharacteristicInvariantsFoliatedBundles}, Heitsch \cite{Heitsch1978--IndependentVariationSecondaryClasses}, and Rasmussen \cite{Rasmussen1980--ContinuousVariationFoliationsCodimensionTwo}, demonstrating surjections like Equation \eqref{eq: characteristic map surjection} in various homology degrees starting at \( 2q+1 \).
Notably, apart from Thurston's original examples, the foliations constructed for this purpose are on manifolds not the sphere, and it is interesting to ask what classes are spherical, or equivalently, when continuously varying foliations can be constructed on spheres.


This question of spherical realization was raised by Hurder in \cite{Hurder1985--ClassifyingSpaceSmoothFoliations}.
In that paper, he provides a positive answer in many cases by using CW- and rational homotopy theory techniques.
He shows in particular that the surjection in Equation \eqref{eq: characteristic map surjection} can be restricted to spherical homology classes, to obtain a surjection
\[ \pi_{2q+1}(B\Gamma^{+}_{q}) \to V_{q} . \]
In other words, there exist families of foliations on \( S^{2q+1} \) that exhibit the independent variation of Gel'fand-Fuks classes shown by previous authors.

We obtain a result similar to Hurder's, for homology classes in \( H_{2q+1}(F\Gamma_{q}) \), as an easy corollary of Theorem \ref{thm: thm intro surgery to sphere} (cf. Corollary \ref{thm: cor continuously varying classes are spherical}).
The construction here is arguably more explicit, depending only on a finite sequence of foliation surgeries for any given homology class in \( H_{2q+1}(F\Gamma_{q}) \).
It also has the advantage that we stay in the smooth category throughout;
if a Haefliger structure \( (M, \F) \) is compatible with the smooth structure of \( M \), in the sense that its Haefliger charts \( M \supseteq U \to \R^{q} \) are smooth functions,\footnote{We work here with smooth Haefliger structures, meaning that the transition functions between charts are smooth, while the charts themselves need only be continuous.} then unobstructed foliation surgeries \( (M', \F') \) can be taken to be compatible with the smooth structure of \( M' \).

\subsection{An explicit application}

In addition to the general argument just described, it is good to have fully explicit examples of foliations to which the theory applies.
To this end, we recall the construction of Thurston \cite{Mizutani2017--ThurstonsConstructionSurjectiveHomomorphismH}, in each codimension \( q \), of foliations whose Godbillon-Vey numbers surject onto \( \R \).
In fact, these are each on a stably trivial \( (2q+1) \)-manifold, but unfortunately, do not have trivial normal bundle.
We explain in Section \ref{sec: Thurston's construction} a modification of the construction that does have trivial normal bundle, so that Theorem \ref{thm: thm intro surgery to sphere} may be applied.
In doing so, we obtain a quite explicit family of continuously varying Godbillon-Vey spheres for each codimension, in Corollary \ref{thm: cor GV spheres}.
To our knowledge, no such explicit construction of Godbillon-Vey spheres has yet appeared in the literature.

\subsection{Previous notions}

Finally, we remark that the idea of foliation surgeries is clearly a natural one, and other authors have made similar constructions in special cases.
For example, Miyoshi in \cite{Miyoshi1982--FoliatedRoundSurgeryCodimensionOneFoliatedManifolds} defines \emph{foliated round surgeries}, for regular codimension-\( 1 \) foliations.
However, the requirement there of preserving regularity of foliation (and in particular, vanishing Euler number of the underlying manifold), means that the underlying manifold surgeries are done on a more complicated attaching regions than spherical surgeries.
There is also the work of Hashiguchi \cite{Hashiguchi2004--SurgeryFormulaDiscreteGodbillonVey}, where regular codimension \( 1 \) foliations on \( 3 \)-manifolds are carried through Dehn surgeries.

The major difference between these approaches and the one taken here is that we make no constraint against singularities.
Giving up on regularity has multiple advantages, both allowing surgeries that would not otherwise be possible, and allowing for a simpler theory.
Indeed, many of the spherical surgeries here have topological obstructions to being regular, and in full generality, singularities are inevitable.
On the other hand, embracing the singularities allowed by Haefliger structures gives substantial flexibility, and the constructions here would be much more intricate if regularity were enforced.
The philosophy here is to do surgery whenever possible, and then in cases where a regularization is not obstructed, the classical results of Thurston \cites{Thurston1976--ExistenceCodimensionOneFoliations, Thurston1974--TheoryFoliationsCodimensionGreaterOne} allow one to argue that the Haefliger structure can be modified to a regular foliation.

\subsection{Acknowledgements}
The author would like to acknowledge Lachlan MacDonald and Michael Albanese for fruitful discussions on the material of this paper.

\section{Basic Notions}\label{sec: basic notions}
\subsection{Haefliger structures}
Although we assume familiarity with the theory of foliations, we begin by recalling some basics to fix notation.
Let \( \Gamma_{q} \) denote the topological groupoid of germs of local \( \mathcal{C}^{\infty} \)-diffeomorphisms of \( \R^{q} \), and \( \Gamma^{+}_{q} \) the subgroupoid of orientation preserving germs.
An element of \( \Gamma_{q} \) is a germ of diffeomorphism \( \gamma \) at source point \( s(\gamma) \in \R^{q} \) and with target point \( t(\gamma) := \gamma(s(\gamma)) \).
The so defined maps \( s,t \colon \Gamma_{q} \to \R^{q} \) are continuous when \( \Gamma_{q} \) is given the \'etale sheaf topology.

\begin{definition}[Haefliger \cite{Haefliger1970--FeuilletagesSurLesVarietesOuvertes}]\label{def: Haefliger structures}
  A codimension-\( q \) \emph{Haefliger structure} or \emph{foliation} \( (M,\F) \) on a topological space \( M \) is the data of an open covering \( \{U_{\alpha}\}_{\alpha \in \mathcal{A}} \) of \( M \), continuous \emph{charts} \( \{f_{\alpha} \colon U_{\alpha} \to \R^{q}\}_{\mathcal{A}} \), and continuous \emph{transition functions} \( \{\gamma_{\beta\alpha} \colon U_{\alpha} \cap U_{\beta} \to \Gamma_{q}\}_{(\alpha,\beta) \in \mathcal{A}^{2}} \).
  This data is assumed to satisfy the following two conditions.
  \begin{enumerate}
  \item
    For each \( x \in U_{\alpha} \cap U_{\beta} \), the relation
    \[ f_{\beta} = \gamma_{\beta\alpha}(x) \circ f_{\alpha} \]
    holds when restricted to sufficiently small neighborhood of \( x \) in \( M \).
  \item
    For each \( x \in U_{\alpha} \cap U_{\beta} \cap U_{\delta} \), and restricted to sufficiently small neighborhood of \( f_{\alpha}(x) \) in \( \R^{q} \), it holds that
    \[ \gamma_{\delta\alpha}(x) = \gamma_{\delta\beta}(x) \circ \gamma_{\beta\alpha}(x) . \]
  \end{enumerate}
  We don't assume maximal atlases; rather, there is the partial order on atlases by refinement, and two Haefliger structures are \emph{equivalent} if they admit a common refinement.

  A Haefliger structure is \emph{regular} if \( M \) is a smooth manifold and all charts are smooth submersions.
\end{definition}

In contrast to some conventions, we will use the term \emph{foliation} interchangeably with the term \emph{Haefliger structure}, so that foliation does not imply regularity.
This is because singularities will be unavoidable in our setting.
Note that for a Haefliger structure \( (M, \F) \), the codimension \( q \) may equal or exceed the dimension of a manifold \( M \).
In particular, the case of a codimension-\( q \) foliation on a manifold of dimension \( q \) may be referred to as a \emph{point foliation}.
For example, the charts of a smooth \( q \)-manifold structure on \( M \) can be regarded as Haefliger charts, comprising a regular point foliation of \( M \) (and conversely).

Given a Haefliger structure \( (M, \F) \) and a continuous map \( f \colon N \to M \), the data of \( \F \) pull back, to give the \emph{pullback Haefliger structure} \( (N, f^{*}\F) \).
In particular, when \( f \) is an inclusion, we may refer to \( f^{*}\F \) as the \emph{restriction to \( N \)} and reuse the notation \( \F \).

\subsection{Concordance classes}
We are primarily interested in foliations up to the following equivalence relation.
\begin{definition}
  A \emph{concordance} between Haefliger structures \( (M, \F_{0}) \) and \( (M, \F_{1}) \) is a Haefliger structure \( \F \) on \( M \times [0,1] \) whose restriction to \( M \times \{i\} \) is equivalent to \( \F_{i} \) for \( i = 0, 1 \).
  Concordance determines an equivalence relation on Haefliger structures, and any two are \emph{concordant} if they are in the same equivalence class.
  In particular, the trivial concordance class is the concordance class containing the \emph{trivial Haefliger structures}, given by any choice of a global Haefliger chart \( f \colon M \to \R^{q} \).
\end{definition}


Before stating the next Lemma, we remark that not all foliations in the trivial concordance class can be covered by a single Haefliger chart;
the regular point foliation of the circle provides such an example.
\begin{lemma}
  A Haefliger structure \( (S^{k}, \F) \) is in the trivial concordance class if and only if it admits an extension \( (D^{k+1}, \tilde{\F}) \).
\end{lemma}
\begin{proof}
  Given an extension \( (D^{k+1}, \tilde{\F}) \), fix an interior ball \( B \) that is sufficiently small enough to be contained in some chart of \( \tilde{\F} \).
  The restriction \( (D^{k+1} \backslash B, \tilde{\F}) \) is a concordance of \( \F \) to a trivial Haefliger structure.

  Conversely, if \( \F \) is in the trivial concordance class, a choice of concordance to trivial Haefliger structure determines a Haefliger structure on the annulus \( S^{k} \times [\tfrac{1}{2},1] \) that restricts to \( \F \) on \( S^{k} \times \{1\} \) and restricts to a trivial Haefliger structure \( f \colon N \to \R^{q} \) on a neighborhood \( N \) of \( S^{k} \times \{\tfrac{1}{2}\} \).
  By standard results, the function \( f \) can be extended smoothly to a function on \( N \cup \tfrac{1}{2}D^{k+1} \).
  This extension gives a Haefliger chart that may be freely added to those on \( S^{k} \times [\tfrac{1}{2}, 1] \) to define the desired extension of Haefliger structure to \( D^{k+1} \).
\end{proof}

\subsection{Classifying spaces and normal bundles}
As \( \Gamma_{q} \) is a topological groupoid, the usual constructions, such as Milnor's infinite join
\cite{Haefliger1971--HomotopyIntegrability}, can be applied to define a classifying space \( B\Gamma_{q} \).
Naturally, \( B\Gamma_{q} \) is characterized by the property that concordance classes of Haefliger structures on \( M \) are in bijection with homotopy classes of maps \( M \to B\Gamma_{q} \).
The space \( B\Gamma_{q} \) can be given a universal Haefliger structure so that the correspondence is given by sending a classifying map to its pullback of the universal structure.

This situation is of course analogous to the case of principal \( G \)-bundles, with \( G \) a topological group.
The primary difference is that concordance is a much stronger notion for principal \( G \)-bundles.
\begin{theorem}[Hussemoller \cite{Husemoller1994--FibreBundles}, Theorem 4.9.8]\label{Thm: thm bundle concordance implies isomorphism}
  Over a paracompact base, concordant principal \( G \)-bundles are isomorphic.
\end{theorem}
No such statement holds for concordant foliations, which may look quite different from each other.

The normal bundle of a Haefliger structure is straightforward from the perspective of classifying spaces.
There is a continuous groupoid map to the general linear group,
\[ \Gamma_{q} \to \GL_{q} , \]
which sends a germ \( \gamma \) to its differential \( \d\!\gamma \) at \( s(\gamma) \), the differential identified as an element of \( \GL_{q} = \GL(T_{0} \R^{q}) \) using the underlying vector space structure of \( \R^{q} \).
By functoriality of the classifying space construction, one has the induced map
\[ \nu \colon B\Gamma_{q} \to \BGL_{q} , \]
which captures reduction to the normal bundle.
\begin{definition}
  The \emph{normal bundle} of a Haefliger structure \( (M, \F) \) classified by \( M \to B\Gamma_{q} \) is the bundle classified by the composition 
  \[ M \to B\Gamma_{q} \to \BGL_{q} . \]
\end{definition}
Of course, in case of a regular foliation, this agrees with the definition of normal bundle as the quotient \( TM / T \F \) by tangents to leaves.
In any case, one may say that a foliation is \emph{transversely oriented} when the normal bundle is oriented, meaning that the classifying map takes values in \( \BGL^{+}_{q} \subset \BGL_{q} \), where \( \GL^{+}_{q} \) is the subgroup of orientation preserving elements in \( \GL_{q} \).
The preimage 
\[ B \Gamma^{+}_{q} = \nu^{-1}(\BGL^{+}_{q}) \subset B \Gamma_{q} \]
classifies the transversely oriented foliations.

The homotopy fiber of the map \( \nu \), denoted \( F\Gamma_{q} \), classifies the Haefliger structures with trivial normal bundle.
A crucial result is the following, due to Mather and then Thurston \cite{Thurston1974--FoliationsGroupsDiffeomorphisms}.
\begin{theorem}\label{thm: Mather-Thurston}
  \( F\Gamma_{q} \) is \( (q+1) \)-connected.
\end{theorem}
For our purposes, this has the useful consequence that for large codimensions, all vector bundles over a sphere can be realized as the normal bundle of a Haefliger structure.
More precisely, from the long exact sequence of the fibration
\[ F\Gamma_{q} \to B\Gamma^{+}_{q} \to \BGL^{+}_{q} , \]
one sees immediately that the maps 
\[ \nu_{*} \colon \pi_{k}(B\Gamma^{+}_{q}) \to \pi_{k}(\BGL^{+}_{q}) \]
are isomorphisms for \( k \le q+1 \) and surjective for \( k = q + 2 \).
It is an interesting question, to which we do not know the answer, whether the homomorphism is a surjection for all \( k \).

Another easy consequence of the Mather-Thurston Theorem is that the classifying space \( B\Gamma^{+}_{q} \) is simply connected.

\subsection{Characteristic classes}
The cohomology classes on \( B\Gamma_{q} \) and \( F\Gamma_{q} \) determine, via pullback by the classifying map, \emph{characteristic classes} of foliations.
For example, is the cohomology class on \( B\Gamma_{q} \) that pulls back to the well known Godbillon-Vey class, whose definition we briefly recall.

Suppose given a regular, transversely orientable, codimension-\( q \) foliation \( (M, \F) \) on smooth manifold \( M \).
By the Frobenius Theorem, this foliation is determined by a choice of decomposable \( q \)-form \( \omega \) on \( M \), for which furthermore there is a \( 1 \)-form \( \eta \) satisfying
\[ \d\omega = \eta\w\omega . \]
The \( (2q+1) \)-form
\[ \eta \w (\d\eta)^{q} \]
is closed, and the cohomology class it represents is independent of all choices just made in the construction.
This cohomology class is the Godbillon-Vey class.

For a general characteristic class of degree \( n \), say \( c \in H^{n}(B\Gamma_{q}) \), and a closed \( n \)-manifold \( M \) with a foliation \( \F \) classified by the map 
\[ f \colon M \to B\Gamma_{q} , \]
we may speak of the \emph{characteristic number}, the integral
\[ c(\F) = \int_{M} f^{*}c \in \R .  \]
As one might expect, these are invariant under foliated cobordism, by the Stoke's Theorem.
This means, in particular, that the characteristic numbers are invariants of concordance class.

\begin{remark}
  We have now two distinct ways to detect nontriviality of a concordance class \( (M, \F) \).

  \begin{enumerate}
  \item
    \emph{Non-vanishing characteristic numbers.}
    The characteristic classes of a trivial foliation on \( M \) vanish.
    So, if \( c(\F) \) is not zero, then the concordance invariance of characteristic numbers demonstrates \( \F \) as not being in the trivial concordance class.

  \item
    \emph{Nontrivial normal bundle.}
    The normal bundle of a concordance of Haefliger structures determines a concordance of normal bundles.
    By Theorem \ref{Thm: thm bundle concordance implies isomorphism}, concordant vector bundles are isomorphic, so if the normal bundle of \( \F \) is not trivial, there can be no concordance of \( \F \) to a trivial Haefliger structure.
  \end{enumerate}
\end{remark}

It is a remarkable fact about foliation characteristic classes that they are not always discrete invariants.
A characteristic class \( c \in H^{k}(B\Gamma_{q}) \) (or in \( H^{k}(F\Gamma_{q}) \)) is said to admit \emph{continuous variation} if the pairing 
\[ c \colon H_{k}(B\Gamma_{q}) \to \R \]
has image containing a non-empty open subset of \( \R \).
More generally, a finite set \( S \) of degree \( k \) characteristic classes is said to admit \emph{independent variation} if the evaluation map 
\[ c_{S} \colon H_{k}(B\Gamma_{q}) \to \R^{S} \]
has image containing a non-empty open set.
It is not difficult to see, by taking disjoint unions of foliations, that independent variation implies that \( H^{k}(B\Gamma_{q}) \) surjects onto all of \( \R^{s} \).

For an example, it is well known that in each codimension-\( q \), the Godbillon-Vey invariant in \( H^{2q+1}(B\Gamma_{q}) \) admits continuous variation.

One can also ask for the stronger condition of \emph{independent variation on \( M \)} for a fixed orientable \( k \)-manifold \( M \), meaning that there exists an open subset \( U \subset \R^{S} \) and family of foliations \( {(M, \F_{\lambda})}_{\lambda \in U} \) for which evaluation against \( S \) gives a surjection to \( U \).
In particular, for the case of \( M \) the sphere \( S^{k} \), this is \emph{independent homotopy variation}, meaning that the evaluation map 
\[ c_{S} \colon \pi_{k}(B\Gamma_{q}) \to \R^{S} \]
has image containing a nontrivial open subset.
In this case again, independent variation implies full surjection to \( \R^{S} \).
This can be seen, for example, using the foliation connected sum, Definition \ref{def: foliation connected sum} below.

\subsection{Surgery on manifolds}\label{sec: surgery on manifolds}
We recall the standard construction of surgery \cite{Milnor1961--ProcedureKillingHomotopyGroupsDifferentiableManifolds}, although we will need a slight variation that adjusts the collar where the surgery happens.
A \emph{\( k \)-surgery} on a manifold \( M \) can be performed for any embedded sphere \( S^{k} \) with trivialized normal bundle,
\[ \varphi \colon S^{k} \times 3 D^{n-k} \hookrightarrow M , \]
with result the \( \varphi \)-surgered manifold \( M' \).
The surgered \( M' \) is obtained by removing the image \( \varphi(S^{k} \times D^{n-k}) \) and smoothly gluing in a new region \( 2 D^{k+1} \times S^{n-k-1} \) by a radial identification
\[ S^{k} \times \left(\Int(2 D^{n-k}) \backslash D^{n-k}\right) \cong S^{k} \times (1,2) \times S^{n-k-1} \cong \left(\Int(2 D^{k+1}) \backslash D^{k+1}\right) \times S^{n-k-1} . \]
In case that \( M \) is oriented, the radial identification should be taken orientation reversing if \( M' \) is to be oriented;
in effect, the inner boundary of the left hand annulus is sent to the outer boundary of the right hand annulus.
The construction can be summarized as in the following diagram.
\begin{equation}\label{diagram: surgery}
  \begin{tikzcd}
    M & \mathring{M} \ar[r, hook] \ar[l, hook'] & M' \\
    S^{k} \times 3 D^{n-k} \ar[u, hook, "\varphi"] & S^{k} \times (1,2) \times S^{n-k-1} \ar[r, hook] \ar[u, hook] \ar[l, hook'] & 2 D^{k+1} \times S^{n-k-1} \ar[u, hook, "\varphi'" swap]
  \end{tikzcd}
\end{equation}
Here the open manifold \( \mathring{M} \) may be regarded as a submanifold of both \( M \) and \( M' \), namely
\begin{equation}\label{eq: common domain of surgering}
  M \backslash \varphi(S^{k} \times D^{n-k})
  = \mathring{M}
  = M' \backslash \varphi'(D^{k+1} \times S^{n-k-1}) .
\end{equation}
The surgery is performed along the collar \( S^{k} \times (1,2) \times S^{n-k-1} \) in \( \mathring{M} \), which in \( M \) is the intersection of \( \mathring{M} \) and the image of \( \varphi \) on \( S^{k} \times 2D^{n-k} \).

The surgery from \( M \) to \( M' \) can be realized by a simple cobordism.
This \emph{trace cobordism}  \( W(M, M') \) can be obtained by gluing onto \( [0,1] \times M \) the handle \( 2 D^{k+1} \times 2 D^{n-k} \) according to the choice of \( \varphi \), so that the boundary component intersecting \( \{1\} \times M \) is diffeomorphic to \( M' \).
The following particularly symmetric construction of \( W(M, M') \) will enable us to easily extend some foliations across the entire trace, when it is possible.
We require the handle attaching to have image in
\[ [0,1] \times \varphi(S^{k} \times 2 D^{n-k}) \subset [0,1] \times M , \]
so  it suffices to describe a model attaching to
\[ [0,1] \times S^{k} \times 2 D^{n-k} .\]
By the linear mapping of \( [0,1] \) onto \( [1,2] \) that fixes \( 1 \) and sends \( 0 \) to \( 2 \), we regard this as the product of annulus by disk,
\[ ([1,2] \times S^{k}) \times 2 D^{n-k} \subset \R^{k+1} \times \R^{n-k} , \]
with associated bi-radial coordinates \( (s, \theta, t, \vartheta) \).
Likewise, regard the handle \( 2 D^{k+1} \times 2 D^{n-k} \) using bi-radial coordinates on \( [0,2] \times S^{k} \times [0,2] \times S^{n-k-1} \), and map it into \( \R^{n+1} \),
\[ \begin{tikzcd}[row sep={8mm,between origins}]
    [0,2] \times S^{k} \times [0,2] \times S^{n-k-1} \ar[r] & \R^{k+1} \times \R^{n-k} \\
    (s, \theta, t, \vartheta) \ar[r, mapsto] & (\rho^{1}(s,t), \theta, \rho^{2}(s, t), \vartheta)
  \end{tikzcd} \]
where
\[ \rho = (\rho^{1}, \rho^{2}) \colon [0,2]^{2} \to [0,2] \times [0,\tfrac{3}{2}]  \]
is any choice of smooth embedding that one obtains for attaching in case \( k = 0, n = 1 \).
In other words, one simply obtains the general attaching model by `spinning' the \( (k,n) = (0,1) \) model about the spheres \( S^{k} \) and \( S^{n-k-1} \).
The \emph{attaching region}, the intersection of \( [0,1] \times M \) and \( 2 D^{k+1} \times 2 D^{n-k} \) is topologically a product of \( S^{k} \) and a contractible domain of dimension \( n-k+1 \).

The key property of this model of attaching is that \emph{the spherical coordinates \( \theta \) on the \( S^{k} \) factor of the attaching region and the handle are identified by the identity map.}

\section{Foliation Surgery}\label{sec: foliation surgery}
\subsection{Definition of foliation surgeries}
We now turn to foliated surgeries.
Suppose that the surgered-from manifold \( M \) of the previous section carries a Haefliger structure \( \F \).
This may be restricted to the common domain of surgery \( \mathring{M} = M \cap M' \) defined in Equation \eqref{eq: common domain of surgering}, to define a partial foliation of the surgered-to \( M' \).
The restriction may or may not be extendable to a foliation on all of \( M' \), much less the trace of surgery.
In case that extensions exist, we make the following definitions.
\begin{definition}
    A \emph{weak foliation surgery of} \( (M, \F) \) is a Haefliger structure \( (M', \F') \) that restricts to \( \F \) on \( \mathring{M} \).

    A \emph{(lax) foliation surgery of} is a Haefliger structure \( (M', \F') \) such that there exists a Haefliger structure on the trace \( W(M,M') \) with boundary restriction \( (M , \F) \cup (M', \F') \).

    A \emph{strong foliation surgery} is \( (M',F') \) that satisfies the conditions to be both a weak and lax foliation surgery.
\end{definition}

Foliation surgeries are clearly not unique.
As well, depending on \( \F \) and the choice of \( \varphi \), there may fail to exist any foliation surgery \( (M', \F') \), in either the weak or lax sense.
The obstructions to existence are described in Definition \ref{def: concordance class obstruction}.

Due to their flexibility, lax foliation surgeries are the most natural choice for our purposes, and we will simply say foliation surgery to mean lax foliation surgery.
The lax definition has also the advantage that it can be made without reference to a choice of model for the surgery \( M' \) or trace \( W(M,M') \).

When extant, lax foliation surgeries can be obtained by first extending \( \F \) to the trace \( W(M, M') \) and then restricting to \( M' \).
One is thus free to include an arbitrary concordance of \( (M, \F) \) `before' performing surgery.
On the other hand, an arbitrary lax foliation can be strictified to a strong surgery;
one may simply modify the foliation of \( W(M, M') \) in its concordance class to be the product foliation \( (M, \F) \times [0,\tfrac{1}{2}] \) on \( M \times [0, \tfrac{1}{2}] \), and then smoothly retract \( W(M,M') \) into itself, onto the union of \( M \times [0, \tfrac{1}{2}] \) and a handle.

Even with their additional flexibility, lax foliation surgeries behave well with respect to foliation characteristic numbers.
Indeed, the next Proposition follows immediately from cobordism invariance.
\begin{proposition}\label{thm: prop lax surgery preserves char numbers}
  Foliation characteristic numbers are unchanged by foliation surgery from  \( (M, \F) \) to
  \( (M', \F') \).
\end{proposition}

In contrast to the Proposition, weak foliation surgeries may fail to preserve characteristic numbers.
In such case, the weak foliation surgery cannot be a strong foliation surgery, the different characteristic numbers on \( (M, \F) \) and \( (M', \F') \) obstructing any extension to the trace.
On the other hand, it will follow from Remark \ref{rem: surgery concordance classes} and Theorem \ref{thm: unobstructed foliation surgery exists} below that a weak foliation surgery implies the existence of a strong foliation surgery.

For an example of a necessarily weak foliation surgery, fix any codimension \( 1 \) foliation \( (S^{3}, \F) \) with non-vanishing Godbillon-Vey number.
Because the foliation is regular and transversely orientable, there exists an embedded circle that is everywhere transverse to the foliation\footnote{In fact, the famous spherical wobble, Figure 1 of \cite{Thurston1972--NoncobordantFoliationsS3}, gives a picture of one such circle; locally, a fiber of the unit tangent bundle of the hyperbolic disk.}.
A thin enough normal neighborhood \( N \) of this circle is a solid torus on which the foliation restricts to define a trivialized normal bundle structure, \( N = S^{1} \times D^{2} \), which can be taken so that the normal fibers \( D^{2} \) are leaves of \( \F \).
The complement of \( N \) is a large solid torus \( \tilde{N} = D^{2} \times S^{1} \), and the surgery on \( \tilde{N} \) is \( M' = S^{1} \times S^{2} \).
But on \( M' \), the restriction of \( \F \) clearly extends to the trivial fiber bundle foliation over \( S^{1} \), which has Godbillon-Vey number \( 0 \).

Of course, the issue here is that the bulk of the foliation was cut out.
It would be interesting to understand more generally the extent to which a continuously varying family of surgeries changes the characteristic numbers.
If it could change continuously, this would provide a direct source of continuous variation.
On the other hand, if it could not, then this would indicate deeper nontrivial topological behavior.

\subsection{The concordance class obstruction}
As already indicated, it is not always possible to make a foliation surgery.
The following definition captures the potential obstruction.
\begin{definition}\label{def: concordance class obstruction}
  For \( (M, \F) \) and attaching map \( \varphi \colon S^{k} \times 3 D^{n-k} \hookrightarrow M \), the \emph{foliation surgery obstruction class} to surgery along \( \varphi \) is the concordance class of the restriction of \( \F \) to the sphere \( \varphi(S^{k} \times \{0\}) \).
\end{definition}

\begin{remark}\label{rem: surgery concordance classes}
  Any map of \( S^{k} \) homotopic in \( M \) to \( \varphi|_{S^{k} \times \{0\}} \) pulls back a foliation concordant to the obstruction class.
  In particular, one obtains the same concordance class by restricting \( \F \) to \( \varphi(S^{k} \times \{p\}) \) for any other point \( p \in D^{n-k} \).
  Consequently, if \( (M, \F) \) admits a weak foliation surgery \( (M', \F') \) along \( \varphi \), then the surgery obstruction class vanishes.
  Indeed, for any \( p \in \partial (2 D^{n-k}) \), the restriction of \( \F' \) to the disk \( \varphi'(2 D^{k+1} \times \{p\}) \) exhibits the concordance class of \( \F \) restricted to \( \varphi(S^{k} \times \{p\}) \) as trivial.
  It holds similarly that the obstruction class vanishes if \( (M, \F) \) admits a lax foliation surgery.

  The converse statement is part of the Theorem \ref{thm: unobstructed foliation surgery exists} below.
\end{remark}

\section{\texorpdfstring{\( 0 \)-Surgeries}{0-Surgeries} and Foliation Connected Sums}\label{sec: foliation connected sums}
Before turning to the problem of existence for general foliation surgeries, we treat the special case of \( 0 \)-surgeries, and in particular, connected sums.
This case is particularly straightforward, because there is no obstruction.
After showing the existence of foliated \( 0 \)-surgeries, we use it to define the group structure on concordance classes of foliations of the sphere.
This group structure will be used in the construction of general foliation surgeries in the following section.

\subsection{Existence}
The existence of foliated \( 0 \)-surgeries can be seen by the following direct construction.

\begin{proposition}\label{thm: prop 0 surgeries exist}
  Given a foliated \( (M, \F) \) and \( \varphi \colon S^{0} \times 3 D^{n} \hookrightarrow M \), a strong foliation surgery \( (M', \F') \) exists.

  If \( (M, \F) \) is transversely oriented, then \( (M',\F') \) can be chosen transversely orientable as well.
\end{proposition}

\begin{proof}
  To simplify matters, we modify \( \F \) slightly in its concordance class, so that it is constant-valued near the basepoints \( H_{0} = \varphi(S^{0} \times \{0\}) \).
  With the notation 
  \[ H_{r} = \varphi(S^{0} \times r D^{n-k}) \]
  for the pair of disks in \( M \) of radius \( r \le 3 \), let
  \[ p \colon H_{3} \to H_{3} \]
  be a smooth, radially defined function that equals the identity in a neighborhood of the boundary and maps both components of \( H_{2} \) to their respective origins in \( H_{0} \).

  There clearly exists a smooth homotopy \( F \) from \( \Id_{H_{3}} \) to \( p \), which may furthermore be arranged so that \( F(t, \cdot) \) remains the identity near the boundary of \( H_{3} \) for all \( t \) and so that \( F(t,x) = p(x) \) for \( t \) close to \( 1 \) and all \( x \in H_{3} \).
  We denote again by \( F \) the homotopy defined on \( M \) via extension by identity,
  \[ F \colon M \times [0,1] \to M . \]
  The desired concordance is then given by the foliation \( \bar{\F} = F^{*}\F \) on \( M \times [0,1] \).
  By construction, this foliation has the property that its restriction to \( H_{2} \times [1-\epsilon,1] \) is contained in a single Haefliger chart \( f \), which is furthermore constant on both components of the domain.
  The choice of \( f \) may be made so as to be globally constant on its domain.

  It is now straightforward to extend the foliation from \( M \times [0,1] \) to the entire surgery trace \( W(M,M') \), only remaining to extend the atlas across the attached handle \( 2D^{1} \times 2D^{n} \).
  The domain of \( f \) intersects the handle at its endpoints, defining a constant function on \( (2 D^{1} \backslash D^{1}) \times 2D^{n} \), which is then extend to the constant function on the full handle.
  This extension defines a chart that is compatible with the original charts on \( M \times [0,1] \), and so coherently defines a Haefliger structure \( \hat{\F} \) on \( W(M,M') \).
  The restriction of \( \hat{\F} \) to \( M' \) gives a lax surgery.

  To obtain a strong surgery, observe that the foliation of \( M' \) agrees with that of \( M \) outside of the image of \( \varphi \).
  Repeating the construction replacing \( \varphi \) with its restriction to \( H_{r} \) for sufficiently small \( r \) will thus produce a strong foliation surgery for the original \( \varphi \).

  For the final claim, suppose that the foliation \( (M,\F) \) is transversely oriented, which holds if and only if \( \F \) admits an atlas of Haefliger charts with all transition functions orientation preserving.
  In such case, the atlas on \( (M',\F') \) constructed here also has orientation preserving transition functions.
  Indeed, the transition functions of \( (M, \F) \) and \( (M', \F') \) are seen to be the same after observing that all intersections of Haefliger charts occur in the common domain \( \mathring{M} = M \cap M' \).
\end{proof}

\begin{remark}
  The construction of Proposition \ref{thm: prop 0 surgeries exist} sometimes applies to general surgeries.
  For \( k > 0 \), if a normal trivialization \( \varphi \colon S^{k} \times D^{n-k} \hookrightarrow M \) can be covered by a single Haefliger chart of \( \F \), then a strong foliation surgery can be constructed using the same idea.
  However, it is not true in general that \( \varphi(S^{k} \times \{0\}) \) can be so covered, even in cases where strong foliation surgeries along \( \varphi \) exist.
\end{remark}

\begin{remark}
  The  extension of chart across the neck is not regular in the given surgery.
  This is so even if the original charts were regular, and one cannot in general expect to do better.
  For example, consider the two dimensional torus equipped with a regular codimension-\( 1 \) foliation.
  The \( 0 \)-surgery at two distinct points will result in the genus \( 2 \) surface, which does not support a regular foliation of codimension \( 1 \).
  Thus, the surgery necessarily introduces singularities in this case.
\end{remark}

\subsection{Uniqueness}
Proposition \ref{thm: prop 0 surgeries exist} provides the existence of foliation surgeries along a zero sphere \( \varphi \), which raises the question of uniqueness, at least up to concordance.
On the one hand, there do exist other foliation surgeries on \( \varphi \), which are strictly weak foliation \( 0 \)-surgeries, and so inequivalent to the strong surgery just constructed; these include topologically nontrivial twisting across the neck, see Remark \ref{rmk: manifold connected sum is not foliation connected sum} for example.
On the other hand, the lax \( 0 \)-surgeries are almost unique up to concordance, and can be made so by fixing a choice of transverse orientation for \( \F \).
This leads to the definition, which will be justified by the following Lemma.

\begin{definition}
  A foliation \( 0 \)-surgery \( (M',F') \) of a transversely oriented foliation \( (M,\F) \) is a \emph{trivial foliation \( 0 \)-surgery} if it is a lax surgery for which the trace foliation \( (W, \bar{\F}) \) admits a transverse orientation extending that of \( \F \).
\end{definition}
Note that the transversely orientable foliations constructed at the end of proof of Proposition \ref{thm: prop 0 surgeries exist} are examples.
In fact, the trivial foliation surgeries are unique up to concordance.
We show this and well-definedness up to concordance in the following Lemma, which will in turn allow us to define the group of foliations on spheres below.

\begin{lemma}\label{thm: lem 0-surgery well defined}
  Given transversely oriented \( (M, \F) \) and \( \varphi \) as in Proposition \ref{thm: prop 0 surgeries exist}, any two trivial foliation \( 0 \)-surgeries along \( \varphi \) are concordant.
  As a consequence, the transversely oriented concordance class of the trivial foliation surgery \( (M', \F') \) depends only on the transversely oriented concordance class of \( (M, \F) \) and the isotopy class of \( \varphi \).
\end{lemma}
\begin{proof}
  For the duration of the proof, let any trivial foliation surgery as constructed in Proposition \ref{thm: prop 0 surgeries exist} be called a \emph{basic surgery}.
  We prove the first claim by demonstrating that any trivial foliation \( 0 \)-surgery is concordant to a basic surgery and then demonstrating that any two basic surgeries are concordant.

  Suppose given a trivial foliation \( 0 \)-surgery \( (M, \F) \) to \( (M',\F') \), and its trace \( (W, \bar{\F}) \).
  On \( W \), fix a Morse function \( h \) equal to \( 0 \) at boundary \( M \), equal to \( 2 \) at boundary \( M' \), and equal to \( 1 \) at a unique critical point in the interior.
  Fix an embedded disk \( D^{1} \times D^{n} \to W \) that extends the attaching map \( \varphi \colon S^{0} \times D^{n} \to M \) and sends \( (0,0) \) to the critical point of \( h \).
  A sufficiently thin neighborhood \( U \) of \( D^{1} \times \{0\} \) can be covered by a single Haefliger chart: \( D^{1} \times \{0\} \) is compact, so can be covered by finitely many charts, and, \( U \) small enough, there is no obstruction to combining them into a single chart one by one.
  By modifying this chart on a smaller neighborhood \( U' \), arrange that \( D^{1} \times \{0\} \) is contained in a single level set.
  Now, the ``pulling open'' used in the proof of Proposition \ref{thm: prop 0 surgeries exist} can be done in a family over \( D^{1} \), to further modify \( \bar{\F} \) so that it is constant on \( U' \) while unchanged outside of \( U \).

  Fix sufficiently small \( \epsilon > 0 \) that for all \( t \le 1 + \epsilon \) the level set \( h^{-1}(t) \) intersects \( U' \).
  The set \( W_{\epsilon} = h^{-1}([0, 1+\epsilon]) \) is diffeomorphic to \( W \), and the restriction \( (W_{\epsilon}, \bar{\F}) \) admits a chart that is constant on \( U' \).
  This means, in particular, that the restriction \( (h^{-1}(1+\epsilon), \bar{\F}) \) is a basic surgery on \( M' = h^{-1}(1+\epsilon) \).
  On the other hand, the restriction of \( \bar{\F} \) to \( h^{-1}([1+\epsilon, 2]) \cong M' \times [0,1] \) shows that this basic surgery is concordant to the given trivial foliation \( 0 \)-surgery \( (M', \F') \).

  For the next step, we require the following collapsing map.
  The Morse function \( h \) has critical level set \( h^{-1}(1) \), which is the topological space \( M^{\vee} \) intermediate between \( M \) and \( M' \), obtained from either of the quotient maps
  \[ c \colon M \to M / \varphi(S^{0} \times \{0\}) \cong M^{\vee} \quad \mbox{ or } \quad c' \colon M' \to M' / \varphi'(\{0\} \times S^{n-1}) \cong M^{\vee} . \]
  These quotients can be extended uniformly to a map
  \[ \bar{c} \colon W \to M^{\vee} \times [0,2] , \]
  as follows.
  Fixing a metric on \( W \) and thus a gradient flow of \( h \), let \( S_{-} \) and \( S_{+} \) be respectively the set of points flowing to or from the critical point.
  The flow of the gradient determines a diffeomorphism
  \[ W \backslash (S_{-} \cup S_{+} ) \cong (M \backslash \varphi(S^{0} \times \{0\})) \times [0,2] , \]
  which has a unique continuous extension \( \bar{c} \), sending each \( p \in S_{-} \cup S_{+} \) to \( (*, h(p)) \in M^{\vee} \times [0,2] \), with \( * \) the singular point of \( M^{\vee} \).

  Now suppose given two basic surgeries \( (M', \F_{1}) \) and \( (M', \F_{2}) \) from \( (M,\F) \), realized by trace foliations \( (W_{1}, \bar{\F}_{1}) \) and \( (W_{2}, \bar{\F}_{2}) \).
  We may glue at the common \( M \)-boundaries, to obtain a foliation on \( W_{1} \cup W_{2} / \sim \), say
  \[ (W_{3}, \F_{3}) . \]
  This defines a cobordism between \( (M', \F_{1}) \) and \( (M', \F_{2}) \), but is not a concordance because \( W_{3} \) is not the product \( M' \times [0,1] \).
  Despite this, after fixing a metric on \( W_{3} \), the collapse maps of \( W_{1}, W_{2} \) glue to a map
  \[ \bar{c} \colon W_{3} \to M^{\vee} \times [0,4] \]
  that collapses the singular flows of the two Morse critical points of \( W_{3} \).
  By assumption, this collapse locus is covered by two constant valued Haefliger charts \( U'_{i} \subset W_{i}  \), which we would like to combine into one on the union \( U'_{3} = U'_{1} \cup U'_{2} \), which is homotopic to \( S^{1} \).
  Both charts being locally constant, the only obstruction to doing so would arise if the normal bundle of \( (U'_{3}, \F_{3}) \) failed to be orientable, but the assumption that both basic surgery traces are chosen compatible with the transverse orientation of \( (M, \F) \) means that the transverse orientations of \( (U_{i}, \F_{i}) \) glue compatibly.
  Granted that the Haefliger structure on \( (M_{3}, \F_{3}) \) has a chart that is constant on \( U'_{3} \), which contains the singular set \( \bar{c}^{-1}(*) \), there exists a foliation \( \F_{4} \) on \( M^{\vee} \times [0,4] \) that pulls back to \( \F_{3} \).
  We have on the other hand the map
  \[ c' \times \Id_{[0,4]} \colon M' \times [0,4] \to M^{\vee} \times [0,4] , \]
  and the pullback of \( \F_{4} \) gives a concordance between the two basic surgeries.

  For the second claim, suppose we start with concordant transversely oriented foliations \( (M, \F_{0}) \) and \( (M, \F_{1}) \), the restrictions of a transversely oriented foliation \( (M \times [0,1], \F_{t}) \).
  Given any trivial surgery \( (M', \F'_{0}) \) of \( \F_{0} \), with trace \( (W, \bar{\F}) \), affix the concordance  \( (M \times [0,1], \F_{t}) \) at the boundary \( M \), to obtain a transversely oriented lax surgery \( (M, \F_{1}) \) to \( (M', \F'_{0}) \).
  This demonstrates \( (M', \F'_{0}) \) as a trivial foliation surgery of \( (M, \F_{1}) \), so by what we have just shown, \( (M', \F'_{0}) \) is concordant to any trivial foliation surgery \( (M', \F'_{1}) \).

  Finally, it is a standard fact, the disc Theorem of Palais \cite{Palais1960--ExtendingDiffeomorphisms}, that the underlying manifold surgery \( M' \) depends only on the isotopy class of \( \varphi \).
  Independence of \( \F' \) from the isotopy class of \( \varphi \) follows from concordance invariance, because one can use the Palais isotopy to define a concordance by pullback.  
\end{proof}

\subsection{The group of foliated spheres}

Given pointed, oriented manifolds \( (M_{1},*), (M_{2},*) \) of equal dimension, let \( M = M_{1} \cup M_{2} \) and
\[ \varphi_{0} \colon S^{0} \to M \]
be the map to basepoints.
For a normal trivialization
\[ \varphi \colon S^{0} \times 3 D^{n} \hookrightarrow M , \]
chosen appropriately with respect to orientation of \( M \), the surgery along \( \varphi \) results in the manifold connected sum \( M' = M_{1} \# M_{2} \).

Now suppose the manifolds are foliated, \( (M_{i}, \F_{i}) \), with transverse orientations fixed, and let 
\[ (M, \F) = (M_{1}, \F_{1}) \cup (M_{2}, \F_{2}) \]
be their disjoint union.
\begin{definition}\label{def: foliation connected sum}
  The \emph{foliation connected sum} \( (M_{1}, \F_{1}) \# (M_{2}, \F_{2}) \) is the concordance class of the trivial foliation \( 0 \)-surgery of \( (M, \F) \) along \( \varphi \).
\end{definition}
By Lemma \ref{thm: lem 0-surgery well defined}, this concordance class is well defined, and depends only on the concordance classes of \( \F_{i} \), their transverse orientations, and the connected components of basepoints.

Recall the wedge sum (or pointed coproduct) of pointed spaces, the topological space \( M_{1} \vee M_{2} \) obtained by disjoint union and then identification of basepoints.
As the wedge sum is a coproduct, the (pointed) classifying maps \( f_{i} \colon M_{i} \to B\Gamma^{+}_{q} \) of \( \F_{i} \) determine a map \( f_{1} \vee f_{2} \), which classifies a Haefliger structure on \( M_{1} \vee M_{2} \).
On the other hand, there is a collapsing map \( M_{1} \# M_{2} \to M_{1} \vee M_{2} \), which sends the neck to the basepoint.
By construction, the concordance class of \( (M_{1}, \F_{1}) \# (M_{2}, \F_{2}) \) is pulled back by this collapse for some Haefliger structure on \( M_{1} \vee M_{2} \).
These Haefliger structures agree up to concordance, so that the following diagram of classifying maps commutes up to homotopy.
\begin{equation}\label{eq: diagram connected and wedge sums}
  \begin{tikzcd}[row sep={6mm,between origins}]
    M_{1} \# M_{2} \ar[dd]\ar[dr, "f_{1} \# f_{2}", pos=0] & \\
    & B\Gamma^{+}_{q} \\
    M_{1} \vee M_{2} \ar[ur, "f_{1} \vee f_{2}" swap, pos=0] &
  \end{tikzcd}
\end{equation}

Now take both manifolds the standard sphere, \( M_{1} = M_{2} = S^{k} \).
In this case, \( S^{k} \# S^{k} \cong S^{k} \), so the foliation connected sum defines an addition law for foliations of the sphere.
Furthermore, the collapse map is in this case \( S^{k} \cong S^{k} \# S^{k} \to S^{k} \vee S^{k} \), which is the co-multiplication map used to define homotopy groups.
For this reason, the addition of concordance classes is the same as the addition of elements in \( \pi_{k}(B\Gamma^{+}_{q}) \).

\begin{definition}
  The \emph{(foliation) concordance class group} 
  \[ \Theta^{k}_{\Gamma^{+}_{q}} = \pi_{k}(B\Gamma^{+}_{q}) \]
  is the set of concordance classes of transversely oriented codimension-\( q \) Haefliger structures on \( S^{k} \), with the group law given by foliation connected sum.
\end{definition}
In this definition it is important that the classifying space \( B\Gamma^{+}_{q} \) is simply connected.
Indeed, the simple connectivity means that pointed homotopy classes of spheres in \( B\Gamma^{+}_{q} \) are the same as free homotopy classes.
For comparison, \( B\Gamma_{q} \) is not simply connected, so one has only the map
\[ \pi_{k}(B\Gamma_{q}) \to \Theta^{k}_{\Gamma_{q}} , \]
where \( \Theta^{k}_{\Gamma_{q}} = [S^{k}, B\Gamma_{q}] \) is the set of free homotopy classes, equal the quotient of pointed homotopy classes \( \pi_{k}(B\Gamma_{q}) = [S^{k}, B\Gamma]_{*} \) by the action of \( \pi_{1}(B\Gamma_{q}) = \Z/2\Z \).
In particular, the group law of \( \pi_{k}(B\Gamma_{q}) \) does not descend to \( \Theta^{k}_{\Gamma_{q}} \) in this case.

It follows immediately from the invariance of characteristic numbers under lax foliations, Lemma \ref{thm: prop lax surgery preserves char numbers}, that the concordance class groups are additively compatible with characteristic numbers.
That is to say, for any characteristic class \( c \in H^{k}(B\Gamma^{+}_{q}, \R) \), integration determines a group homomorphism
\[ c \colon \Theta^{k}_{\Gamma^{+}_{q}} \to \R , \]
which sends a foliation \( (S^{k}, \F) \) to its characteristic number \( c(\F) \).
Furthermore, given an arbitrary oriented \( k \)-manifold \( M \), the foliation connected sum defines an action of \( \Theta^{k}_{\Gamma^{+}_{q}} \) on the collection \( [M, B\Gamma^{+}_{q}] \) of transversely oriented codimension-\( q \) Haefliger structure concordance classes on \( M \), and the characteristic map
\[ c \colon [M, B\Gamma^{+}_{q}] \to \R \]
is clearly compatible with this action.
As a consequence, whenever a set of degree \( k \)-characteristic classes \( S \) varies independently on \( S^{k} \), the set also varies independently for Haefliger structures on \( M \).
Applying this observation to \( M = S^{k} \) itself, it is easily seen that whenever the characteristic map \( c_{S} \) of a set \( S \) has open image in \( \R^{S} \), it is actually surjective.
Remark here that allowing singularities helps us here;
not all manifolds admit regular foliations of a given codimension, but there do always exist Haefliger structures, such as the trivial one.

\begin{remark}\label{rmk: manifold connected sum is not foliation connected sum}
  A small word of caution is in order, on the difference between manifold and foliation connected sum in the case of point foliations.
  Recall that these codimension-\( n \) Haefliger structures on an \( n \)-manifold \( M \) are comparable with smooth structures of \( M \);
  the charts of the smooth structure on \( M \) may be regarded as the charts of a Haefliger structure \( \F_{M} \), resulting in a regular point foliation.
  For two manifolds, the manifold connected sum \( M \# N \) is again a smooth manifold, but the induced Haefliger structure \( \F_{M \# N} \) may not agree with the foliation connected sum \( (M, \F_{M}) \# (N, \F_{N}) \).
  
  This can already be seen for the regular point foliation \( (S^{2}, \F_{S^{2}}) \), whose normal bundle is the tangent bundle of \( S^{2} \) and thus has Euler number \( 2 \in \Z \cong H^{2}(S^{2}) \).
  On the one hand, \( S^{2} \) is the identity element for manifold connected sums of smooth spheres (which is in any case the trivial group in this dimension).
  On the other hand, \( (S^{2}, \F_{S^{2}}) \) cannot be the identity in \( \Theta^{2}_{\Gamma^{+}_{2}} \), because the first Chern class \( c \) of \( B\Gamma^{+}_{2} \), pulled back from \( \BGL_{2} \) by the normal map, gives a group morphism
  \[ c \colon \Theta^{2}_{\Gamma^{+}_{2}} \to \R \]
  that sends \( (S^{2}, \F_{S^{2}}) \) to \( 2 \).

  Alternatively, one may note that any trivial Haefliger structure on \( S^{2} \) represents the identity element of \( \Theta^{2}_{\Gamma^{+}_{2}} \), and \( (S^{2}, \F_{S^{2}}) \) is not concordant to the trivial foliation, due to its nontrivial normal bundle.
  
  The point is that the manifold connected sum is required to result in a smooth structure, so it is necessary to include a spherical inversion across the neck in its definition.
  This is different from the foliation connected sum, which is homotopically trivial at the neck.

  One could of course combine these two operations, to the larger group of foliations on all smooth structures of \( S^{k} \), but we will not require such considerations here.
\end{remark}


As a final remark of this subsection, we note that the fact that the operation of foliation connected sum defines a group law on foliated spheres can be checked directly, without resort to the group structure of \( \pi_{k}(B\Gamma^{+}_{q}) \).
For example, given \( (S^{k}, \F) \), the inverse is \( (\bar{S}^{k}, \F) \), where \( \bar{S}^{k} \) is the same sphere but with opposite orientation.
To see this explicitly, take the round \( S^{k} \subset \R^{k+1} \) with basepoint at \( e_{1} = (1,0,\ldots) \) and identify with \( \bar{S}^{k} \) via the reflection orthogonal to \( e_{1} \).
Let \( p \colon S^{k} \to D^{k} \) be the orthogonal projection along \( e_{1} \).
The connected sum \( (S^{k}, \F) \# (\bar{S}^{k}, \F) \) is concordant to a foliation that can be pulled back by \( p \) from a foliation on \( D^{k} \), and a null-concordance of the foliation on \( D^{k} \) pulls back to a null-concordance of \( (S^{k}, \F) \# (\bar{S}^{k}, \F) \).

\subsection{A transitive action on bundle trivializations over \texorpdfstring{\( S^{k} \)}{the k sphere}}
We now turn to the second main aim of this section, which is to demonstrate the following Lemma.

Given a transversely oriented foliation \( (S^{k}, \F) \) in the trivial concordance class, the normal bundle \( E_{0} \) of \( \F \) is necessarily trivial, and has an orientation induced by the transverse orientation of \( \F \).
However, depending on \( k \), the trivializations of \( E_{0} \) may not be unique, even up to homotopy.
By contrast, the normal bundle \( E_{1} \) of a transversely oriented foliation \( (D^{k+1}, \tilde{\F}) \) does admit a unique oriented trivialization, up to homotopy.
Given a choice of oriented trivialization \( \sigma \) of \( E_{0} \), say that an extension \( (D^{k+1}, \tilde{\F}) \) of \( \F \) is compatible with \( \sigma \) if \( E_{1} \) admits a trivialization that restricts to \( \sigma \).

\begin{lemma}\label{thm: lem transitive action on trivializations}
  For \( k \le q+1 \), fix a transversely oriented codimension-\( q \) foliation \( (S^{k}, \F) \) in the trivial concordance class.
  For each choice of oriented trivialization \( \sigma \) of \( E_{0} \), there exists an extension \( (D^{k+1}, \tilde{\F}) \) compatible with \( \sigma \).
\end{lemma}

The proof of the Lemma relies on the action of \( \pi_{k+1}(\BGL^{+}_{q}) \) on the oriented trivializations of \( E_{0} \), which we recall briefly.
To lighten notation, let \( G \) be an arbitrary connected Lie group, which in the application to normal bundles will be \( \GL^{+}_{q} \).
First recall that, while the trivial \( G \)-bundle \( E_{0} \) is isomorphic to \( S^{k} \times G \), there is no canonical choice of trivializing isomorphism.
Rather, there is the set of homotopy trivializations
\[ \Triv(E_{0}) = \left\{ (E_{0}, *) \stackrel{\cong}{\to} \left(S^{k}\times G, (*,e)\right) \right\}/\sim , \]
with equivalence \( \sim \) by homotopy through pointed \( G \)-isomorphisms, where we have fixed a basepoint for \( E_{0} \) and \( S^{k} \).
It is easily seen that this set is in bijection with homotopy classes of pointed sections of \( E_{0} \), and this latter set is itself acted on freely and transitively by \( \pi_{k}(G) = [S^{k}, G]_{*} \), the difference between two sections given by pointwise multiplication.
With the identification \( \pi_{k}(G) \cong \pi_{k+1}(BG) \), we have thus a free and transitive action of \( \pi_{k+1}(BG) \) on \( \Triv(E_{0}) \), which we will call the \emph{standard action} on trivializations.

A second description of this action can be given using bundle connected sums.
Suppose given two principal \( G \)-bundles \( E, E' \) over respective \( n \)-manifolds \( M, M' \), the bundles classified by respective maps \( f \colon M \to BG \) and \( f' \colon M \to BG \).
If one fixes basepoints of \( M \) and \( M' \) and supposes that \( f, f' \) are pointed maps to \( BG \), then there is a well defined pointed map \( M \vee M' \to BG \).
But under the assumption that \( G \) is connected, we have \( BG \) simply connected, and so the sets \( [M,BG] \) and \( [M, BG]_{*} \) (respectively,  free and pointed homotopy classes of maps) are in bijection (cf. Hatcher \cite{Hatcher2002--AlgebraicTopology} Proposition 4.A.2.).
In this case, it is unambiguous to fix arbitrary basepoints of \( M, M' \) and form the wedge sum \( f \vee f' \), to obtain the classifying map
\[ M \# M' \to M \vee M' \stackrel{f \vee f'}{\to} BG \]
of a \emph{bundle connected sum} \( E \# E' \) over \( M \# M' \).

For a more explicit construction, suppose without loss of generality that the disks \( D, D' \) in \( M \) and \( M' \) over which \( M \# M' \) is formed are trivializing disks for \( E \) and \( E' \) respectively.
For the covering of \( M \# M' \) by 
\[ U = M \backslash \tfrac{1}{2}D \quad \mbox{ and } \quad U' = M' \backslash \tfrac{1}{2}D' ,\]
we have the restriction bundles \( E|_{U} \) and \( E'|_{U'} \).
The intersection \( U \cap U' \) is the neck of \( M \# M' \), and we have by assumption that the restrictions \( E|_{U} \) and \( E'|_{U'} \) are trivialized.
We obtain the bundle \( E \# E' \) by gluing these trivialized bundles over the neck.
Because \( G \) is connected, all choices of gluing are isotopic, and the resulting bundle is unique up to isomorphism.

The bundle connected sum is topologically compatible with the classifying maps for \( G \)-bundles in precisely the same manner as Diagram \ref{eq: diagram connected and wedge sums}, but mapping into \( B G \).
Again, taking both base manifolds as the sphere, one obtains an additive operation on bundles over the sphere, which is clearly the same as the usual addition in \( \pi_{k}(BG) \).


Return to the action on trivializations of \( E_{0} \).
For \( E_{1} \) the trivial \( G \)-bundle over \( D^{k+1} \), let \( \Map(E_{0}, E_{1}) \) be the set of homotopy classes of bundle maps from \( E_{0} \to E_{1} \) over the boundary inclusion \( i \colon S^{k} \to D^{k+1} \).
There is an isomorphism
\[ \Map(E_{0}, E_{1}) \cong \Triv(E_{0}) , \]
which sends a map \( E_{0} \to E_{1} = D^{k+1} \times G \) over \( i \) to the restriction of codomain \( E_{0} \to S^{k} \times G \).

Fix a trivialization \( t \colon E_{0} \to E_{1} \) and an element \( \alpha \in \pi_{k+1}(BG) \), regarded as the map \( \alpha \colon S^{k+1} \to BG \) classifying a bundle \( E_{\alpha} \) over \( S^{k+1} \).
Over a point in the interior of \( D^{k+1} \), form the bundle connected sum \( E_{1} \# E_{\alpha} \), which is again isomorphic to \( E_{1} \).
By construction, the intersection \( E_{1} \cap (E_{1} \# E_{\alpha}) \) is a bundle over a neighborhood of \( S^{k} \) in \( D^{k+1} \).
As such, the map \( t \) defines a map \( \hat{t} \colon E_{0} \to E_{1} \cap (E_{1} \# E_{\alpha}) \), for which the codomain may be extended to \( E_{1} \# E_{\alpha} \), to define a map \( \alpha \cdot t \).
In other words, there is for each trivialization \( t \) a unique choice of trivialization \( \alpha \cdot t \) making the following diagram commute.
\[ \begin{tikzcd}
    & E_{1} \cap (E_{1} \# E_{\alpha}) \ar[dl] \ar[dr] & \\
    E_{1} \ar[d] & E_{0} \ar[r, "\alpha\cdot t"] \ar[u, "\hat{t}"] \ar[l, "t" swap] \ar[d] & E_{1} \# E_{\alpha} \cong E_{1} \ar[d] \\
    D^{k+1} & S^{k} \ar[r] \ar[l] & D^{k+1}
  \end{tikzcd} \]
This defines an action of \( \pi_{k+1}(B G) \) on \( \Map(E_{0}, E_{1}) \) by \( t \mapsto \alpha \cdot t \).
The difference between the two trivializations of \( E_{0} \) is precisely the clutching map of \( \alpha \), and it follows that this action is the same as the standard action on trivializations.

Specializing again to the case of normal bundles, let \( \Theta^{k}_{\GL^{+}_{q}} = \pi_{k}(\BGL^{+}_{q}) \) denote the group of \( \GL^{+}_{q} \) bundles over \( S^{k} \).
The normal map \( B\Gamma^{+}_{q} \to \BGL^{+}_{q} \) naturally induces a group homomorphism
\[ \Theta^{k}_{\Gamma^{+}_{q}} \to \Theta^{k}_{\GL^{+}_{q}} , \]
which can be interpreted as the fact that the connected sum of foliations is compatible with the connected sum of their normal bundles.
It is an immediate consequence of the Mather-Thurston Theorem \ref{thm: Mather-Thurston} that this is an isomorphism for \( k \le q+1 \) and a surjection for \( k = q + 2 \).
It is an interesting question, to which we do not know the answer, whether the homomorphism is a surjection for all \( k \).

In any case, these consideration make it simple to show the Lemma.
\begin{proof}[Proof of Lemma \ref{thm: lem transitive action on trivializations}]
  Fix an arbitrary oriented extension of \( (S^{k}, \F) \to (D^{k+1}, \tilde{\F}) \), extant by assumption.
  The normal bundle of this extension determines a trivialization \( t \colon E_{0} \to E_{1} \).
  Any other oriented trivialization of \( E_{0} \) is of the form \( \alpha \cdot t \) for some \( \alpha \in \pi_{k+1}(\BGL^{+}_{q}) \).
  By the Mather-Thurston Theorem, and assumption that \( k \le q+1 \), there exists a foliation \( \F_{\alpha} \) of \( S^{k+1} \) whose normal bundle has classifying map \( \alpha \).
  The foliation \( (D^{k+1}, \tilde{\F}) \# (S^{k+1}, \F_{\alpha}) \) induces the trivialization \( \alpha \cdot t \) of \( E_{0} \), and the Lemma follows.
\end{proof}

\section{General Foliation Surgeries}\label{sec: general foliation surgeries}
\subsection{Existence of surgery foliations}
Using the results of the previous section, we can characterize the existence of foliation surgeries in general.
\begin{theorem}\label{thm: unobstructed foliation surgery exists}
  Given a foliated manifold \( (M, \F) \) and an attaching map
  \( \varphi \colon S^{k} \times 3 D^{n-k} \hookrightarrow M \),
  a strong foliation surgery \( (M', \F') \) along \( \varphi \) exists if and only if the obstruction concordance class of \( \varphi \) is trivial.

  When this holds, if \( \F \) is transversely oriented, then the surgery can be chosen to be transversely oriented.
  Furthermore, if the normal bundle of \( \F \) is trivial and \( k \le q+1 \), then the strong foliation surgery can be chosen so that the normal bundle of \( \F' \) is again trivial.
\end{theorem}
\begin{proof}
  By Remark \ref{rem: surgery concordance classes}, the existence of even a weak foliation surgery implies that the obstruction class of \( \varphi \) vanishes.

  Supposing conversely that the obstruction class
  vanishes, we construct a lax foliation surgery.
  This can be done using the same idea as was used to prove Proposition \ref{thm: prop 0 surgeries exist}, but now in a bundle over \( S^{k} \).
  For any \( r \le 3 \), let \( H_{r} = \varphi(S^{k} \times r D^{n-k}) \) be the bundle of radius \( r \) disks in \( M \), and
  \[ p \colon H_{3} \to H_{3} \]
  be a smooth function that equals the identity in a neighborhood of the boundary and maps each fiber of of \( H_{2} \) to its basepoint in \( H_{0} = \varphi(S^{k} \times \{0\}) \cong S^{k} \).
  There exists a smooth homotopy
  \[ F \colon M \times [0,1] \to M \]
  that is the identity outside of \( H_{3} \) and restricted to \( H_{3} \) interpolates \( \Id_{H_{3}} \) to \( p \), arranged so that \( F(t,x) = p(x) \) for \( t \) close to \( 1 \) and all \( x \in H_{3} \).

  Use \( F \) to extend \( \F \) to a foliation \( \bar{\F} = F^{*}\F \) on the product region \( M \times [0,1] \) of the surgery trace \( W = W(M, M') \).
  This foliation has the property that its restriction to \( H_{2} \times [1-\epsilon,1] \) factors through the core \( (H_{0}, \F) \): for the projection \( \pi \colon H_{2} \times [1-\epsilon,1] \to H_{0} \times \{1\} \), it holds that
  \[ \bar{\F}|_{H_{2} \times [1-\epsilon,1]} = \pi^{*} \F . \]
  The assumption that the obstruction class of \( \varphi \) vanishes means there exists an extension
  of \( (H_{0}, \F) \) to some foliation \( (D^{k+1}, \tilde{\F}) \).
  We construct from this a larger extension \( (2 D^{k+1}, \tilde{\F}) \) such that \( \tilde{\F} \) is radially constant near the boundary;
  more precisely, we require that 
  \[ \tilde{\F}|_{A} = r^{*}\F \]
  for the annulus \( A = [1-\epsilon,2] \times S^{k} \) in \( 2 D^{k+1} \) and the retraction \( r \colon A \to S^{k} \).
  Such an extension can be constructed by pulling back \( (D^{k+1}, \tilde{\F}) \) by any smooth radially symmetric map \( 2 D^{k+1} \to D^{k+1} \) that fixes the origin and projects a neighborhood of the annulus \( A \) to \( S^{k} \).
  We then extend this to a foliation on the handle \( (2 D^{k+1} \times 2 D^{n-k}, \tilde{\F}) \) by pulling back via the projection onto first factor.

  The trace cobordism \( W \) is covered by \( M \times [0,1] \) and the handle \( 2 D^{k+1} \times 2 D^{n-k} \).
  We have given a foliation on both, and supposing \( \epsilon > 0 \) sufficiently small, these are seen to cohere into a single Haefliger structure.
  Indeed, on their intersection, both foliations factor through the projection to the same foliation \( (S^{k}, \F) \).
  
  Finally, observe that the foliation on \( M' \) just constructed agrees with that on \( M \) outside of the image of \( \varphi \).
  Repeating the construction replacing \( H_{3} \) with \( H_{r} \) for sufficiently small \( r \) will thus produce a strong foliation surgery for the original \( \varphi \).

  In the case that \( (M, \F) \) has a transverse orientation, there is an induced transverse orientation on \( (S^{k}, \F) \).
  If the extension \( (D^{k+1}, \tilde{\F}) \) is given the compatible transverse orientation, then \( (M', \F') \) has an induced transverse orientation.
  
  Now, suppose that the normal bundle of \( \F \) is trivial and that \( k \le q+1 \).
  We will show that an appropriate choice of extension \( (D^{k+1}, \tilde{\F}) \) can be made so that the normal bundle \( \bar{\nu} \) of \( \bar{\F} \) over \( W \) is trivial.
  This suffices to complete the proof, because the restriction \( (M', \F') \) has also trivial normal bundle in this case.

  To this end, continue with the same notation as above, with \( \tilde{\F} \) to be determined.
  The situation is summarized in the commuting diagram, with all arrows maps of Haefliger structure.
  
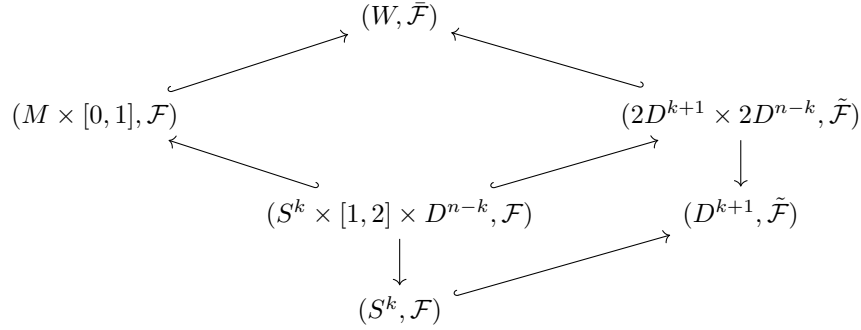
\[
    \begin{tikzcd}
      & (W, \bar{\F}) & \\
      (M \times [0,1], \F) \ar[ur, hook] & & (2 D^{k+1} \times 2 D^{n-k}, \tilde{\F}) \ar[ul, hook'] \ar[d] \\
      & (S^{k} \times [1,2] \times D^{n-k}, \F) \ar[ul, hook'] \ar[d] \ar[ur, hook] & (D^{k+1}, \tilde{\F}) \\
      & (S^{k}, \F) \ar[ur, hook] &
    \end{tikzcd}
  \]
  For any choice of \( \tilde{\F} \), the normal bundle of \( \bar{\F} \) can be described by a clutching construction, because its restrictions to both of \( M \times [0,1] \) and \( 2D^{k+1} \times 2 D^{n-k} \) are trivial.
  So, a choice of trivialization of the normal bundle of \( \F \) induces a trivialization of the normal bundle over the attaching region \( S^{k} \times [1,2] \times D^{n-k} \simeq S^{k} \), and the normal bundle of \( \bar{\F} \) will be trivial if this trivialization is the same as that induced by \( \tilde{\F} \).
  But we have shown in Lemma \ref{thm: lem transitive action on trivializations} that, supposing \( k \le q+1 \), all trivializations can be realized by some choice of \( \tilde{\F} \).
\end{proof}

\subsection{Foliation surgeries to the sphere}

Recall that an \( n \)-manifold \( M \) is \emph{stably trivial} if its tangent bundle is stably trivial, meaning there is a trivial line bundle \( \epsilon_{1} \) so that \( TM \oplus \epsilon_{1} \) is the trivial bundle.
These are particularly nice manifolds from the perspective of surgery theory;
by work of Milnor \cites{Milnor1961--ProcedureKillingHomotopyGroupsDifferentiableManifolds, Milnor1960--CobordismRingOmegaComplexAnalogue} and Wall \cite{Wall1959--NoteCobordismRing},  there exists a sequence of surgeries to the sphere,
\begin{equation}\label{eq: surgery sequence to sphere}
  M \rightsquigarrow M_{1} \rightsquigarrow M_{2} \rightsquigarrow \ldots \rightsquigarrow S^{n} .
\end{equation}
The procedure is to first surger along \( 0 \)-spheres to obtain to a connected manifold, then along generators of the fundamental group to obtain a simply connected manifold, and so on.
In particular, the surgeries in this case need only be done on spheres of dimension up to half that of \( M \), after which Poincar\'e duality guarantees that one has the sphere.

Although the following Theorem is stated more generally, the case \( k = 2q+1 \) is the one of primary interest, because that is the degree in which examples of independent variation of characteristic classes are known.
\begin{theorem}\label{thm: foliations surgerable to spheres}
  On a stably trivial manifold \( M \) of dimension \( n \le 2q+2 \), every \mbox{codimension-\( q \)} foliation \( (M, \F) \) with trivial normal bundle can be surgered to a foliation \( (S^{n}, \F') \) with the same characteristic numbers as \( \F \).
\end{theorem}

\begin{proof}
  Proceeding inductively along the surgery sequence \eqref{eq: surgery sequence to sphere}, suppose that \( (M, \F) \) is \( (k-1) \)-connected (with \( k \le q+1 \)) and that \( \F \) has trivial normal bundle.
  Any choice of basic element in \( \pi_{k}(M) \) can be represented by an embedding \( \varphi_{0} \colon S^{k} \to M \) admitting a trivial normal neighborhood \( \varphi \colon S^{k} \times D^{n-k} \to M \), and surgery along \( \varphi \) kills this element of \( \pi_{k}(M) \).
  Because the normal bundle of \( \F \) is trivial, its restriction to \( \varphi_{0} \) will also have trivial normal bundle.
  But then the Mather-Thurston Theorem applies, and the concordance obstruction \( (\varphi_{0}(S^{k}), \F) \) is trivial.
  Thus, Theorem \ref{thm: unobstructed foliation surgery exists}, there is a foliation surgery \( (M', \F') \) along \( \varphi \)  that preserves characteristic numbers and such that \( \F' \) again has trivial normal bundle.
\end{proof}


Although Theorem \ref{thm: foliations surgerable to spheres} only applies directly to stably trivial manifolds, it is sufficient to conclude that for \( n \le 2q+2 \), every set of characteristic classes in \( H^{n}(F\Gamma_{q}) \) admitting independent variation on some manifold also admits independent homotopy variation.
This can be seen using the following classical fact.
\begin{theorem}\label{thm: realization of homology classes}
  Given an orientable \( n \)-manifold \( M \), there exists a stably trivial \( n \)-manifold \( N \) that dominates \( M \).
\end{theorem}
To say that \( N \) \emph{dominates} \( M \) means that \( M \) and \( N \) are orientable manifolds of equal dimension and there exists a smooth map \( f \colon N \to M \) of positive degree, in the sense that \( f \) pushes forward the fundamental class of \( N \) to a positive multiple of the fundamental class of \( M \).
Given a foliation on \( (M, \F) \), the pullback by dominating \( f \) gives a foliation on \( N \) with characteristic numbers scaled by \( \deg(f) \).
Thus an independently varying family of foliations on \( M \) produces a varying family on \( N \), which can furthermore be surgered to the sphere, and we have the next Corollary.
\begin{corollary}\label{thm: cor continuously varying classes are spherical}
  Fix a codimension \( q \) and dimension \( n \le 2q+2 \).
  Any set \( S \) of foliation characteristic classes in \( H^{n}(F\Gamma_{q}) \) that are independent on some manifold \( M \) are also independent on \( S^{n} \).
\end{corollary}

\begin{proof}[Remark on the proof of Theorem \ref{thm: realization of homology classes}]
  Much more generally, in Thom's work \cite{Thom1954--QuelquesProprietesGlobalesVarietesDifferentiables} on Steenrod's manifold realization problem for rational homology classes, the realizing manifold can be chosen stably trivial (see, for example, \cite{Ajij-Chakraborty-Sen2025--ContactDomination}).
  More recent work of Gaifullin provides a more constructive argument, potentially useful for applications.
  In \cite{Gaifullin2013--CombinatorialRealisationCyclesSmallCovers}, Gaifullin shows that in each dimension \( n \) there exists a manifold \( M^{n}_{0} \) (the Tomei manifold) such that every orientable \( n \)-manifold \( M \) is dominated by some finite cover of \( M^{n}_{0} \).

  One description of \( M^{0}_{n} \) is as the collection of real symmetric tridiagonal \( (n+1) \times (n+1) \) matrices with eigenvalues equal a fixed set of \( n+1 \) distinct reals.
  Because this manifold can be realized as the preimage of a regular value for a smooth function \( \R^{2n+1} \to \R^{n+1} \), it has stably trivial tangent bundle.
  The same then holds for its finite covers, and so the claim holds.
  The author would like to credit the discussion by users Igor Belegradek and mme (Mike Miller Eismeier) on MathOverflow question \cite{BelegradekEismeier2020--CanEveryManifoldDominatedParallelizableOne} for the observation that \( M^{0}_{n} \) is stably trivial.
\end{proof}

\section{A Variation on Thurston's Construction of Godbillon-Vey manifolds}\label{sec: Thurston's construction}


We describe in this section the construction of some foliations for which Theorem \ref{thm: foliations surgerable to spheres} can be applied directly.
The result of this application will be for each \( q > 1 \) a family of foliations on \( S^{2q+1} \) whose Godbillon-Vey numbers surject onto \( \R \).
The construction is due to Thurston, as recounted by Mizutani in \cite{Mizutani2017--ThurstonsConstructionSurjectiveHomomorphismH}, although we require a small variation to ensure trivial normal bundle.

Recall that given a smooth bundle \( E \) with (manifold) fiber \( F \), a foliation \( (E, \F) \) is a \emph{bundle foliation} if \( \F \) has codimension equal the dimension of \( F \) and the leaves of \( \F \) are everywhere transverse to the fibers of \( E \).
Let \( \Sigma \) be an arbitrary hyperbolic surface, orientable of genus at least \( 2 \), and \( T^{q-1} \) the torus of dimension \( q-1 \).
\begin{theorem}[Thurston \cite{Mizutani2017--ThurstonsConstructionSurjectiveHomomorphismH}]
  For each integer \( q \ge 1 \) and each \( r \in \R \), there exists a codimension-\( q \) bundle foliation \( \F_{r} \) on the trivial \( S^{q} \)-bundle \( (\Sigma \times T^{q-1}) \times S^{q} \) with Godbillon-Vey number \( GV(\F_{r}) = r \).
\end{theorem}

For such foliations, the normal bundle of \( \F_{r} \) is identified with the pullback of \( T S^{q} \), as is immediate from the triviality of the bundle \( (\Sigma \times T^{q-1}) \times S^{q} \) and the transversality of \( \F_{r} \) to fibers.
The manifold \( (\Sigma \times T^{q-1}) \times S^{q} \) is parallelizable, and for \( q = 1, 3, 7 \) the normal bundle of \( \F \) is trivial, so Theorem \ref{thm: foliations surgerable to spheres} applies as is.
For all other codimensions \( q \), the sphere \( S^{q} \) is not parallelizable, so Theorem \ref{thm: foliations surgerable to spheres} does not apply.

To treat the remaining codimensions \( q \) (which we may suppose to be at least 2), we have the following variation.
A key step of Thurston's construction is the decomposition of fiber \( S^{q} \) as the boundary of \( D^{2} \times D^{q-1} \),
but by choosing a manifold other than \( D^{q-1} \), the same construction can be made, potentially with a parallelizable fiber.
Specifically, fix a compact \( (q-1) \)-manifold \( N \) with non-empty boundary and embeddable into \( \R^{q-1} \), and let \( F \) be the smooth manifold obtained by rounding corners of \( \partial(D^{2} \times N) \).
The point then is that \( N \) can chosen so that \( F \) is parallelizable, and so one obtains foliations that may be surgered to the sphere.
For example, with \( N = D^{2} \times T^{q-3} \), one has \( F \cong S^{3} \times T^{q-3} \).
\begin{theorem}\label{thm: GV foliations general fiber}
  For \( N, F \) as just described, and each \( r \in \R \), there exists a codimension-\( q \) bundle foliation \( \F_{r} \) on \( (\Sigma \times T^{q-1}) \times F \) with Godbillon-Vey number \( GV(\F_{r}) = r \).
\end{theorem}
With the choice of \( N = D^{2} \times T^{q-3} \), we have the immediate Corollary.
\begin{corollary}\label{thm: cor GV spheres}
  For each codimension \( q \ge 1 \) and each \( r \in \R \), there exists a smooth Haefliger structure \( (S^{2q+1}, \F_{r}) \) such that \( GV(\F_{r}) = r \).
\end{corollary}

We recall the essential points of Thurston's construction, and what modifications must be made to accommodate the more general fiber, referring to  \cite{Mizutani2017--ThurstonsConstructionSurjectiveHomomorphismH} for further details.

Let \( N = D^{q-1} \) and \( F = S^{q} \) for the immediate discussion, to be generalized below.
Give \( B = \Sigma \times T^{q-1} \) the metric of constant curvature \( -1 \) and \( 0 \) on its respective factors, and consider the homogeneous action by isometries on \( B \) of the group
\[ G = \SL(\R^{2}) \times \R^{q-1} . \]
Thurston prescribes a large family of suitable Lie algebra maps
\[ \psi_{f} \colon \sl(\R^{2}) \times \R^{q-1} \to \X(F) , \]
depending on a function 
\[ f \colon \partial N = S^{q-2} \to \R^{q-1} . \]
These integrate to group homomorphisms
\[ \Psi_{f} \colon \SL(\R^{2}) \times \R^{q-1} \to \Diff(F) , \]
and each defines a foliation on \( G \times F \) by taking the level sets of the function
\[ \begin{tikzcd}[row sep={6mm,between origins}]
    G \times F \ar[r] & F \\
    (b, p) \ar[r, mapsto] & \Psi_{f}(b) p
  \end{tikzcd} . \]
The function, and thus the foliation, is easily seen to be invariant (equivariant) for respectively the right (left) actions
\[ (b',p)\cdot b = (b' b, \Psi_{f}(b^{-1}) \cdot p) \quad \mbox{ and } \quad  b \cdot (b', p) = (b b', p) . \]
By quotient, first on the right by \( \SO(\R^{2}) \), and then on the left by a cocompact lattice, Thurston arrives at the foliation \( (\Sigma \times T^{q-1} \times F, \F_{f}) \).

The Godbillon-Vey number of a trivial-bundle foliation, such as these are, is particularly amenable to calculation, being determined by the Thurston cocycle \cite{Mizutani1988--GodbillonVeyCocycleDiffR}.
This is a Chevalley-Eilenberg cohomology \( (q+1) \)-cocycle \( \beta \) on \( \X(F) \), 
given by the formula
\[ \beta(X_{0}, \ldots, X_{q}) = \int_{F} (\div X_{0}) \d (\div X_{1}) \w \ldots \w \d (\div X_{q}) , \]
where the divergence of vector fields is given relative a choice of volume form \( \omega \) by
\[ (\div X) \omega = L_{X} \omega , \]
\( L_{X} \) the Lie derivative.
The Thurston cycle is related to the Godbillon-Vey number by the horizontal lift map of the bundle foliation,
\[ m \colon T B \to \X(F) , \]
and the formula
\begin{equation}\label{eq: GV and Thurston cocycle}
  GV(B \times F, \F) = \int_{B} m^{*}\beta .
\end{equation}

The homogeneity of the base and of the foliation \( \F_{f} \) means that \( m^{*}\beta \) need only be evaluated on a single fiber.
Thurston does this, and finds that the the integral is a positive constant times a simpler integral,
\begin{equation}\label{eq: volume of f}
  GV(B \times F, \F) = C \int_{F} f^{*} \omega_{q-1} ,
\end{equation}
where \( \omega_{q-1} \) is the volume divergence form on \( \R^{q-1} \),
\[ \omega_{q-1} = \left(\sum_{i = 0}^{q-1} x^{i} \partial_{x^{i}}\right) \lhk \d x^{1} \w \ldots \w \d x^{q-1} . \]
By the Stoke's Theorem, the integral in \eqref{eq: volume of f} equals the (signed) volume enclosed by \( f \), and this volume can clearly be chosen arbitrarily.

We now describe Thurston's Lie maps \( \psi_{f} \), as well as their generalization to other fibers \( F \).
The action of the first factor, \( \sl(\R^{2}) \to \X(F) \), proceeds through a certain action on the disk \( D^{2} \).
In representation by matrices, \( \sl(\R^{2}) \) determines a linear action on \( \R^{2} \), as well as an action on \( S^{1} \) (the action on oriented lines through \( 0 \) in \( \R^{2} \)), and Thurston's action of \( \sl(\R^{2}) \to \X(D^{2}) \) interpolates these.
Explicitly, in polar coordinates \( (r, \theta) \) on \( \R^{2} \), one may write an  \( \sl(\R^{2}) \)-triple
\begin{align*}
  Y & = \cos(\theta)\sin(\theta) g(r) r \partial_{r} - \sin^{2}(\theta) \partial_{\theta} , \\
  H & = (\cos^{2}(\theta) - \sin^{2}(\theta)) g(r) r \partial_{r} - 2\cos(\theta) \sin(\theta) \partial_{\theta}, \\
  X & = \cos(\theta)\sin(\theta) g(r) r \partial_{r} + \cos^{2} (\theta) \partial_{\theta} ,
\end{align*}
where a tamping of the radial component is accomplished by choice of a smooth function
\[ g \colon \R_{+} \to \R  \]
that is constantly equal to \( 1 \) near \( 0 \) and vanishes away from a small neighborhood of \( 0 \).
The triple determines an \( \sl(\R^{2}) \)-representation that agrees with the linear action near the origin and with the circle action near the boundary.

By assumption, there is a decomposition of \( F \) as a union of two open sets,
\begin{equation}\label{eq: F decomposition}
  F = \partial(D^{2} \times N) \cong (D^{2} \times \partial N) \cup (S^{1} \times N) ,
\end{equation}
and one can define the action \( \sl(\R^{2}) \to \X(F) \) that acts by the \( D^{2} \) action on the first factor of the first component and the circle action on the first factor of the second component.

The action of the \( \R^{q-1} \)-factor on \( \X(F) \) is where the choice of \( f \) enters.
Given smooth
\[ f = (f^{i}) \colon N \to \R^{q-1} , \]
fix the vector fields
\[ Z^{i} := f^{i} g(r) r \partial_{r}  \]
on \( D^{2} \times \partial N \).
These vanish near the boundary of \( D^{2} \times \partial N \), so can be trivially extended to \( F \).
They furthermore commute with each other and the fields \( Y, H, X \), so define the desired Lie map \( \psi_{f} \) that sends \( Y, H, X \) as above and a basis \( t^{i} \) of \( \R^{q-1} \) to \( Z^{i} \).

To compute the divergences of the vector fields it is convenient to take the following volume form on \( F \),
\[
\omega = \left\{
  \begin{array}{rl}
    G(r) \d r \w \d \theta \w \omega_{\partial N} & \mbox{on } D^{2} \times \partial N \\
    \d \theta \w \omega_{N} & \mbox{on } S^{1} \times N
  \end{array} \right.
\]
for fixed volume forms \( \omega_{\partial N}, \omega_{N} \) on \( \partial N \) and \( N \).
Here \( G \colon \R \to \R \) is the identity map near \( 0 \) and constantly \( 1 \) away from \( 0 \).
With this choice, one finds the divergences
\begin{align*}
 \div m(H) & = (h-2)\cos(2\theta) , \\
 \div m(X) & = \tfrac{1}{2}(h-2)\sin(2\theta), \\
 \div m(t^{i}) & = f^{i}h ,
\end{align*}
where
\[ h(r) := \div(g(r) r \partial_{r}) = r g' + 2 g . \]
Note that these are well defined functions because \( h \) equals \( 2 \) near \( \{0\} \times \partial N \subset D^{2} \times \partial N \)  and equals \( 0 \) outside of a small neighborhood of \( \{0\} \times \partial N \).

From this it is straightforward to calculate that
\begin{multline*}
  \beta\left(m(H), m(X), m(t^{1}), \ldots, m(t^{q-1})\right) \\
  = \pm \int_{D^{2}} \tfrac{1}{2}h^{n-2} (h-2)^{2} \left(\cos(2\theta)\right)^{2} \d h \d \theta
  \int_{\partial N} f^{*} \omega_{q-1} .
\end{multline*}
As Thurston finds, the integral over \( D^{2} \) is non-zero, and the integral over \( \partial N \) is proportional to the signed volume of an extension of \( f \) to \( N \), so can be chosen arbitrarily by varying \( f \).

\section{Unclasping the Normal Bundle}\label{sec: unclasping the normal bundle}
As seen in Section \ref{sec: foliation surgery}, a nontrivial normal bundle can lead to foliation surgery obstructions.
We wish to make clear that this obstacle is not always insurmountable, so describe a method of modifying the given foliation in one very special case.
This modification changes the obstruction class along a surgery sphere \( \varphi \), and in good cases, can eliminate it.
One could of course improve on the example here, but we leave that to future works.

The setup we suppose is this.
Let \( (M, \F) \) be a foliation of codimension \( q \), and fix a surgery attaching map \( \varphi \colon S^{q} \times 3 D^{n-q} \).
As we have seen, with these dimensions, the obstruction to surgery is determined by nontriviality of the normal bundle of \( \varphi^{*} \F \).
Now suppose the existence of a compact leaf \( L \) that intersects \( \varphi_{0}(S^{k}) \) in precisely one point.
Make the further assumption that the given leaf \( L \) has trivial holonomy.
By the Reeb stability Theorem (e.g. \cite{Candel-Conlon2000--Foliations} Theorem 2.4.1), there is a neighborhood of \( L \),
\[ N = L \times D^{q} \subset M \]
such that the foliation restricts to the trivial product foliation, with leaves given as the level sets of the function 
\[ f \colon N \to D^{q} . \]
We may regard \( f \) as a chart of the Haefliger structure \( \F \).

All of this granted, the idea is modify the Haefliger chart \( f \) to `unclasp' the normal bundle of \( \varphi^{*}\F \).
In fact, this can easily be done by the transitive action of connected sums on trivializations.
Decompose \( \varphi_{0}(S^{q}) \) into two disks, \( D = N \cap \varphi_{0}(S^{q}) \) and \( D' \) a neighborhood of \( \varphi_{0}(S^{q}) \backslash D \).
The restriction of \( \F \) to \( D \cap D' \simeq S^{q-1} \) is in the trivial concordance class, and the restriction \( D' \) determines a trivialization of its normal bundle.
By the transitive action Lemma \ref{thm: lem transitive action on trivializations}, there exists a choice of extension \( \tilde{\F} \) to \( D \) that determines the same trivialization as \( D' \).
The pullback by \( f \) of \( (D, \tilde{F}) \) gives a new Haefliger structure \( \tilde{\F} \) on \( N \), compatible with \( \F \) near the boundary.
Replacing \( f \) with this new structure, we have a Haefliger structure \( \F' \) on \( M \), for which the surgery obstruction class \( \varphi_{0}^{*}\F' \) has normal trivial bundle, and so vanishes.

Although the setup here is special, it can occur quite naturally in the setting of bundle foliations.
Suppose that \( M \to B \) is a compact bundle with fiber \( S^{q} \) and that \( \F \) is a bundle foliation on \( M \).
For a small disk in \( D \subset B \), the local trivialization \( \varphi \colon D \times S^{q} \hookrightarrow M \) determines a trivially normalized surgery sphere.
It is clear that the foliation surgery obstruction class of \( \varphi \) is nontrivial in case \( q \neq 1, 3, 7 \), because its normal bundle is isomorphic to \( T S^{q} \).
If the bundle holonomy action \( h \colon \pi_{1}(B) \to \Diff(S^{q}) \) of this foliation  has a fixed point \( p \), then the leaf \( L = L_{p} \) through \( p \) determines a section of \( M \), and so intersects the fiber \( S^{q} \) precisely once, at \( p \).
If furthermore the holonomy of this leaf is trivial, then we are in the setup just described.

Many of the classical examples of variation on characteristic classes \cites{Heitsch1978--IndependentVariationSecondaryClasses,Mizutani2017--ThurstonsConstructionSurjectiveHomomorphismH} are close to this, given essentially by prescribing the holonomy of a bundle foliation, which often has fiber the sphere.
The examples frequently have have points in the fiber that are fixed by the holonomy action.
However, in those examples, the holonomy of the leaf through that point is not typically trivial, which is why the construction given here does not apply immediately.

Finally, we remark that this procedure can be done so as to preserve characteristic numbers.
This follows from the localization of characteristic classes, the Chern-Weil map of differential graded algebras from the Gel'fand-Fuks algebra \( WO_{q} \) to differential forms on \( M \), which is defined by a choice of Bott connection and metric on the normal bundle of \( \F \).
One can choose the connection \( \nabla \) and metric \( \epsilon \) on \( \nu_{\F} \) to be flat on \( N \), so that the image forms are identically zero on \( N \).
Likewise, choose the connection \( \nabla' \) and metric \( \epsilon' \) on \( \nu_{\F'} \) flat on \( N \) and agreeing with the choices for \( \nu_{\F} \) on the complement of \( N \).
This done, the characteristic forms clearly agree, and so their characteristic numbers do as well.
Finally, a remark that \( \F' \) may have singularities, but in this case it is still possible to make `adapted' choices of the geometry \( (\nabla', \epsilon') \) such that the Chern-Weil map to forms is well defined, see \cite{Macdonald-Mcmillan2025--ChernWeilTheoryHaefligerSingularFoliations}.

\bibliography{GVspheres}
\end{document}